\documentclass[11pt]{amsart}
\usepackage{amsmath,amsthm,amssymb}
\usepackage{mathrsfs}
\usepackage[hypertexnames=false]{hyperref}
\usepackage{enumitem} 
\usepackage{url}
\usepackage[all]{xy}

\theoremstyle{plain}
\newtheorem{prop}{Proposition}
\newtheorem{thm}[prop]{Theorem}
\newtheorem{thm*}[prop]{Theorem}
\newtheorem{conj}[prop]{Conjecture}
\newtheorem{cor}[prop]{Corollary}
\newtheorem{lemma}[prop]{Lemma}

\theoremstyle{remark}
\newtheorem{remark}[prop]{Remark}

\newcommand{\Z}{{\mathbb Z}}
\newcommand{\N}{{\mathbb N}}
\newcommand{\G}{{\mathbb G}}
\newcommand{\GG}{{\mathbb G}}
\newcommand{\cD}{{\mathcal D}}

\newcommand{\PP}{{\mathbb P}}
\newcommand{\ZZ}{{\mathbb Z}}

\newcommand{\A}{{\mathbb A}}

\newcommand{\cA}{\mathcal{A}}
\newcommand{\cX}{\mathcal{X}}

\newcommand{\cV}{\mathcal{V}}

\newcommand{\cC}{\mathcal{C}}
\newcommand{\cM}{\mathcal{M}}
\newcommand{\br}{\mathbf{r}}
\newcommand{\Gm}{\mathbb{G}_{\mathbf{m}}}

\newcommand{\fM}{\mathfrak{M}}
\newcommand{\fC}{\mathfrak{C}}

\DeclareMathOperator{\Spec}{Spec}

\DeclareMathOperator{\Aut}{Aut}
\DeclareMathOperator{\GL}{GL}
\DeclareMathOperator{\st}{st}
\renewcommand{\ss}{\mathrm{ss}}

\newcommand{\qs}{\mathrm{qs}}
\renewcommand{\bar}{\overline}
\newcommand{\sing}{\mathrm{sing}}

\renewcommand{\tilde}{\widetilde}

\begin{document}
\title[Equivariant versal deformations of semistable curves]{Equivariant versal deformations of \\semistable curves}
\author[Alper]{Jarod Alper}
\author[Kresch]{Andrew Kresch}
\address[Alper]{Mathematical Sciences Institute,
Australia National University,
Canberra, ACT 0200}
\address{Humboldt-Universit\"at zu Berlin, Institut F\"ur Mathematik, Unter den Linden 6, 10099 Berlin} 
\email{jarod.alper@math.anu.edu.au}

\address[Kresch]{
Institut f\"ur Mathematik \\
Universit\"at Z\"urich \\
Winterthurerstrasse 190 \\
CH-8057 Z\"urich
} \email{andrew.kresch@math.uzh.ch}

\thanks{The first author was partially supported by the Alexander von Humboldt Foundation.
The second author was partially supported by the Swiss National Science Foundation}

\begin{abstract}  We prove that given any $n$-pointed prestable curve $C$
of genus $g$ with linearly reductive automorphism group $\Aut(C)$,
there exists an $\Aut(C)$-equivariant miniversal deformation of $C$ 
over an affine variety $W$.
 In other words, we prove that the algebraic stack 
$\fM_{g,n}$ parametrizing $n$-pointed prestable curves of genus $g$
 has an \'etale neighborhood of $[C]$ isomorphic to the quotient stack
 $[W / \Aut(C)]$.
\end{abstract}

\maketitle

\section{Introduction}
\label{sec:intro}

A fundamental question in algebraic geometry is to understand the 
relationship between arbitrary algebraic stacks and quotient stacks. 
While not every algebraic stack is a quotient stack (\cite{ehkv} and 
\cite{kresch-flattening}), it is natural to conjecture that every algebraic
stack is \'etale locally a 
quotient stack around a point with linearly reductive stabilizer.  
Precisely, we formulate the conjecture as follows:

\begin{conj} \label{conj1} Let $X$ be an algebraic stack locally 
of finite type over an algebraically closed field $k$ 
with separated and quasi-compact diagonal
such that 
$X$ has affine stabilizer groups at all closed points.  
Suppose $x \in X(k)$ has a linearly reductive stabilizer group scheme $G_x$.  
Then there exists an affine scheme $W$ over $k$ 
with an action of $G_x$, a $k$-point $w \in W$, 
and an \'etale, representable morphism
$$f \colon [W/G_x] \to X$$
such that $f(w) = x$ and $f$ induces an isomorphism of stabilizer groups at $w$. 
\end{conj}

Conjecture \ref{conj1} after replacing $W$ 
with an algebraic space is a particular case of
the conjecture stated in \cite{alper-structure}. 
Similar questions were raised in \cite[\S 5]{conrad-dejong} and \cite[\S 2]{oprea-localization}.

This conjecture implies that \'etale-local properties of general 
algebraic stacks (satisfying the above hypotheses)
can be inferred from properties of algebraic stacks of the form $[\Spec(A)/G]$ 
with $G$ linearly reductive.  Such quotient stacks are particularly well understood; in particular,
many geometric properties of $[\Spec(A)/G]$ can be related to properties of the GIT
quotient $\Spec(A^G)$.
Additionally, as suggested by Rydh, 
it is possible to attach to an algebraic stack $X$ satisfying Conjecture \ref{conj1}
at a point $x\in X(k)$ a Henselian localization $\mathcal{O}^h_{X,x}$ which is a comodule algebra
over the Hopf algebra of $G_x$ such that $[\Spec(\mathcal{O}^h_{X,x})/G_x]\to X$ satisfies
analogous properties to the usual Henselization $\Spec(\mathcal{O}^h_{W,w})\to W$.

Conjecture \ref{conj1} is known to have a positive answer if 
$X$ has quasi-finite diagonal (e.g., $X$ is a Deligne-Mumford stack) 
or if $X = [U / G]$ where $U$ is a normal scheme and
$G$ is a linear algebraic group
(see \S \ref{subsec:localDM}--\ref{subsec:slicing}).
The purpose of this article is to verify Conjecture \ref{conj1} in 
an interesting and natural moduli problem which 
does not fall into one of the above cases.  
Let $\fM_{g,n}$ be the 
moduli stack of {\it prestable} curves 
(proper flat families of connected nodal curves) of 
genus $g$ with $n$ marked points.

\begin{thm} \label{main-theorem}
  Conjecture \ref{conj1} holds for $\fM_{g,n}$ for all $g, n \ge 0$.  
\end{thm}

This theorem implies that given any $n$-pointed prestable curve $C$
of genus $g$ with linearly reductive automorphism group $\Aut(C)$, 
there exists an affine variety $W$ with an action of $\Aut(C)$ fixing a point $w \in W$
 and a miniversal deformation $\cC \to W$ of $C \cong \cC_w$
 such that there is an 
 action of $\Aut(C)$ on the total family $\cC$ compatible with the action on $W$
 and restricting to the natural action of $\Aut(C)$ on $\cC_w$.

Let $\fM_{g,n}^{\ss} \subset \fM_{g,n}$ be the open substack consisting of
semistable curves (i.e., pointed curves $(C, \{p_i\}_{i=1}^n)$ such that $\omega_C(\sum_i p_i)$ 
has non-negative degree on every component, where $\omega_C$ is the dualizing sheaf).
Since a prestable curve with a linearly reductive automorphism group,
which is not a $0$-pointed smooth curve of genus $0$, is semistable, 
Theorem \ref{main-theorem} reduces to proving that $\fM_{g,n}^{\ss}$ 
satisfies Conjecture \ref{conj1}.

The algebraic stack $\fM_{g,n}^{\ss}$ has particularly exotic properties in connection to
Conjecture \ref{conj1}. 
For instance, $\fM_{g,n}^{\ss}$ has a finite-type open substack which
is not a global quotient stack,
does not have quasi-affine diagonal, and does not admit a good moduli space 
(see \S \ref{subsec:bad-properties}).  For these reasons,
$\fM_{g,n}^{\ss}$ serves as natural test cases for Conjecture \ref{conj1}.

Finally, the technique employed to prove Theorem \ref{main-theorem},
based on stacks of log structures, reveals features
of the stack of semistable curves which may be of independent interest.
For instance, Lemma \ref{lem:localstr} gives a description
of the local structure of the stack of semistable curves near a curve whose
stabilization has just one node, in terms of a particular
zero-dimensional smooth algebraic stack which has appeared in many
settings in algebraic geometry and has been studied in detail by
Abramovich, Cadman, Fantechi, and Wise \cite{acfw}.
In particular, this provides a concrete description of the fiber of
the stabilization morphism over a stable curve with one node.  

The proof of Theorem \ref{main-theorem} proceeds by a 
sequence of reductions.
In Section \ref{sec:semistable}, we show that it suffices to prove
that $\fM_{g}^{\ss}$ satisfies Conjecture \ref{conj1} for all $g \ge 0$
(Proposition \ref{prop:first}); in that section, we also exhibit some exotic properties
of the stack $\fM_g^{\ss}$.
In Section \ref{sec:quotient}, we state a result
(Proposition \ref{prop:finiteetale}) that
reduces Conjecture \ref{conj1} for $\fM_{g}^{\ss}$ to showing the existence of
particular kinds of finite covers of \'etale neighborhoods.
Just this is enough to verify Conjecture \ref{conj1}
in the special case (Theorem \ref{main-theorem-special-case}) of a pair of smooth curves of
distinct positive genera joined by a chain of rational curves
(in Section \ref{sec:specialcase}).
After presenting some stacks generalities in Section \ref{sec:stack}
and a local construction around a strictly semistable curve in
Section \ref{sec:local}, the next task will be to formulate and prove,
in Section \ref{sec:locXg},
a structure result (Proposition \ref{prop:fiberdiagramxtriangle})
for the moduli of semistable curves
with stabilization contracting a single chain of rational curves to a node.
Finally, Theorem \ref{main-theorem} is proved in Section \ref{sec:proof}.

\subsection*{Conventions} In this paper we work with general algebraic stacks
(not required to have separated or quasi-compact diagonal) as defined in
\cite{stacksproject}.
An algebraic stack possessing an \'etale cover by a scheme is called a
Deligne-Mumford stack.
Sheaves of monoids arise in the treatment of log structures;
all (sheaves of) monoids are commutative.

\subsection*{Acknowledgements}  We are grateful for the valuable suggestions provided by the referee.  We thank Rahul Pandharipande for stimulating discussions which motivated this investigation.  We also thank the referee for useful comments and Jack Hall for providing comments on the first draft.

\section{Stable and semistable curves} \label{sec:semistable}

Fix an algebraically closed field $k$ and $g \ge 2$.  Let $\fM_{g,n}$ be the 
moduli stack of {\it prestable} curves 
(i.e., nodal, connected, proper curves) of 
genus $g$ with $n$ marked points, with stabilization morphism
\[ \st\colon \fM_{g,n}\to \bar{\cM}_{g,n},\]
which is flat,
to the stack $\bar{\cM}_{g,n}$ of stable curves of genus $g$ with $n$ marked points.
The algebraic stack $\fM_{g,n}$ is quasi-separated and
locally of finite type over $k$.
Let $\fM_g^{\ss} \subset \fM_g$ denote the locus of \emph{semistable} 
curves, that is, curves whose dualizing sheaf has non-negative 
multidegree. (A similar definition can also be made for $\fM_{g,n}^{\ss} \subset \fM_{g,n}$.)
We will also consider $\fM_g^{\qs}\subset \fM_g^{\ss}$, the locus of
\emph{quasistable} curves, that is, semistable curves where the
\emph{exceptional components} (smooth rational components on which the degree of
the dualizing sheaf is zero) are pairwise disjoint.

We recall, in the context of algebraic stacks of finite presentation over $k$,
a \emph{global quotient stack}, or just \emph{quotient stack}, is
a stack quotient $[U/G]$ for the action of a linear algebraic group $G$ on
a finite-type algebraic space $U$.
Choosing a faithful representation $G\to \GL_n$, we have
\begin{equation}
\label{eqn:UGGLn}
[U/G]\cong [U\times^G\GL_n/\GL_n],
\end{equation}
where $U\times^G\GL_n$ denotes the quotient of $U\times \GL_n$ by $G$, acting
as given on $U$ and by left translation on $\GL_n$, which is an
algebraic space.
So, the definition of quotient stack is unchanged if we restrict to
$\GL_n$-actions.
According to \cite[Lem.\ 2.12]{ehkv}, a quotient stack can be characterized by
the existence of a vector bundle with faithful actions of
the geometric stabilizer group schemes.
Returning to \eqref{eqn:UGGLn}, we remark that the projection
from $U\times^G\GL_n$ to $G\backslash GL_n$ is \'etale locally the projection from a product
with $U$.
When $G$ is reductive, then,
$U$ affine implies $U\times^G\GL_n$ affine.

\subsection{Boundary components}
\label{subsec:boundary-Mg}
Let $\bar{\cC}_g\to \bar{\cM}_g$ be the
universal family over the moduli stack of 
Deligne-Mumford curves of genus $g$.
The algebraic stacks $\bar{\cC}_g$ and $\bar{\cM}_g$ 
are smooth over $\Spec(k)$.  We will denote
the relative singular locus by
$$D_g:=\bar{\cC}_g^{\sing}$$
which is defined by the
Fitting ideal of the sheaf of relative differentials.  
The algebraic stack $D_g$ is smooth of
codimension $2$ in $\bar{\cC}_g$
(cf.\ \cite[\S 1]{deligne-mumford}) and is the normalization 
of the boundary divisor of $\bar{\cM}_g$.  

The irreducible 
components of the boundary divisor of $\bar{\cM}_g$ are 
indexed by unordered pairs of 
positive integers summing to $g$ and an additional 
element (labelled ``irr'' in \cite{acg}).

There is a degree 2 \'etale cover
$\widehat{D}_g \to D_g$
where $\widehat{D}_g$ parametrizes
stable curves together with a node and a choice of tangent direction
to the curve at the node.
We recall from \cite[XII.10.11]{acg} that
$$\widehat{D}_g = \bar{\cM}_{g-1,2} \sqcup \big( \bar{\cM}_{1,1} 
\times \bar{\cM}_{g-1,1} \big) \sqcup  \big( \bar{\cM}_{2,1} 
\times \bar{\cM}_{g-2,1}  \big) \sqcup \cdots \sqcup 
\big( \bar{\cM}_{g-1,1} \times \bar{\cM}_{1,1} \big),$$
where the morphism to $D_g$ is given by
by gluing sections.

\subsection{Bad properties of $\fM_g^{\ss}$} \label{subsec:bad-properties}
We start by exhibiting a few exotic properties of  $\fM_g^{\ss}$
which indicate that  $\fM_g^{\ss}$ is a particularly interesting candidate
to test the validity of Conjecture \ref{conj1}.
In fact, we restrict our attention to the finite-type open substack
$\fM_g^{\qs}$ of quasistable curves.

The inclusion $i \colon \bar{\cM}_g \hookrightarrow \fM_g^{\qs}$
has complement of codimension 2, and it follows that
pullback and pushforward by stabilization $\st^*$ and $\st_*$ give an 
equivalence of categories of vector bundles.  Since every vector bundle on $\fM_g^{\qs}$ 
is the pullback of a vector bundle
on $\bar{\cM}_g$, there is no vector bundle on $\fM_g^{\qs}$ with
faithful action on the fiber by the stabilizer at a strictly semistable curve.
So, by the characterization of global quotient stacks in terms of
vector bundles recalled above,
$\fM_g^{\qs}$ is not a global quotient stack.  
In fact, the argument of \cite[Prop.\ 5.2]{kresch-flattening} may be adapted to
establish the stronger statement
that $\fM_g^{\qs}$ does not have quasi-affine diagonal.

Now suppose $g\ge 3$ and
consider the fiber $F$ of $\st \colon \fM_g^{\qs} \to \bar{\cM}_g$ over 
a curve $C'$ with a single node and trivial automorphism group.
The fiber $F$ consists of 
two curves, $C'$ as well as the strictly semistable curve 
$C$ obtained by inserting a $\PP^1$ at the node.  
Let $U$ be a nodal cubic curve in $\PP^2$.
We now argue that the fiber $F$ may be identified with $[U / \Gm]$.
Let $(\tilde{C}', p_1, p_2)$ be the pointed normalization of $C'$ and consider the trivial family $\cX = \tilde{C}' \times \PP^1 \to \PP^1$ with sections $s_1,s_2$ corresponding to $p_1, p_2$.  Here $\GG_m$ acts on $\PP^1$ (and on $\cX$) in the standard way.  Let $\tilde{\cX}$ be the equivariant blowup of $\cX$ at both $p_1$ in the fiber over $0$ and $p_2$ in the fiber over $\infty$, and let $\tilde{s}_1, \tilde{s}_2$ be the proper transform of the sections $s_1,s_2$.   Now glue the sections $\tilde{s}_1, \tilde{s}_2$ to construct a family $\tilde{C} \to \PP^1$ such that the fibers $\tilde{C}_0$ and $\tilde{C}_{\infty}$ are isomorphic to $C$ but the generic fiber is $C'$.  Finally, since the fibers $\tilde{C}_0$ and $\tilde{C}_{\infty}$ are equivariantly isomorphic respect to {\it opposite} actions of $\GG_m$, we may glue these two fibers to construct a $\GG_m$-equivariant family of curves $\cD$ over the nodal cubic $U$.  One checks that the induced map from $[\cD / \GG_m]$ to the fiber $F$ of $\st \colon \fM_g^{\qs} \to \bar{\cM}_g$ over $C'$ is an isomorphism
(it is representable, is a monomorphism, and satisfies the valuative
criterion for properness).
It follows that there is no open substack $\cV$ of $\fM_g^{\qs}$ 
containing $[C]$ and admitting a good moduli space.

\subsection{First reduction}
We show that in order to establish Theorem \ref{main-theorem}, it suffices to show that $\fM_g^{\ss}$ satisfies Conjecture \ref{conj1}.

\begin{prop}  \label{prop:first}
If $\fM_{g}^{\ss}$ satisfies Conjecture \ref{conj1} for all $g \ge 0$, then $\fM_{g,n}$
satisfies Conjecture \ref{conj1} for all $g,n \ge 0$.
\end{prop}

\begin{proof}
An $n$-pointed prestable genus $g$ curve
$(C, \{p_i\}_{i=1}^n)$ with $g$ or $n$ positive or $C$ singular
is semistable (i.e. $\omega_C(\sum_i p_i)$ 
has non-negative multidegree) if the automorphism group $\Aut(C, \{p_i\}_{i=1}^n)$ is linearly reductive.  Therefore, to establish that $\fM_{g,n}$
satisfies Conjecture \ref{conj1} it suffices to show that the moduli stack $\fM_{g,n}^{\ss}$ of pointed semistable curves satisfies Conjecture \ref{conj1}.

For $i=1, \ldots, n$, let $C_1, \ldots, C_n$ be automorphism-free smooth 1-pointed curves 
of distinct genera $g_1, \ldots, g_n$ greater than $g$. Let $g' = g+ g_1 + \cdots + g_n$.  The morphism 
$$\mathfrak{M}_{g,n}^{\ss} \to \mathfrak{M}_{g'}^{\ss},$$
defined by attaching $C_i$ to the $i$th marked point, is a closed immersion.  
If $\mathfrak{M}_{g'}^{\ss}$ satisfies Conjecture \ref{conj1}, 
then so does $\mathfrak{M}_{g,n}^{\ss}$.
\end{proof}

\section{Quotient structure of algebraic stacks} \label{sec:quotient}

\subsection{Stabilizer preserving morphisms}
\label{subsec:stabpres}
A morphism $f\colon X\to Y$ of algebraic stacks
is \emph{stabilizer preserving} at a given
geometric point of $X$ if it induces an isomorphism of 
stabilizer group schemes at that geometric point.  
We say that $f \colon X \to Y$ is \emph{pointwise stabilizer preserving} if it 
is stabilizer preserving at all geometric points.  

\begin{lemma} \label{lem:etale}
Let $f\colon X\to Y$ be an \'etale representable morphism of
algebraic stacks. If the fiber of $f$ over every geometric point consists of a single point, then $f$ is an
isomorphism.
\end{lemma}

\begin{proof}  Since $f$ is \'etale and surjective, to show that $f$ is an isomorphism
it suffices to show that projection $X \times_Y X \to X$ is an isomorphism.  The relative
diagonal $X \to X \times_Y X$ is an isomorphism (as it is a surjective open immersion), and therefore
so is $X \times_Y X \to X$ as the composition $X \to X \times_Y X \to X$ is the identity map.
\end{proof}

\begin{lemma}
\label{lem:finiteinertia}
Let $W$ and $X$ be algebraic stacks with
finite stabilizer groups at geometric points, and let
$f\colon W\to X$
be a separated morphism.
If $X$ has finite inertia, then so does $W$.
\end{lemma}

\begin{proof}
The inertia $I_W\to W$ factors through
$I_X\times_XW$, with morphism
$I_W\to I_X\times_XW$ obtained via base
change from the relative diagonal of $f$, hence finite, and
$I_X\times_XW\to W$
finite by the hypothesis on $X$.
\end{proof}

\begin{lemma}
\label{lem:representablelocus}
Let $X$ be a Deligne-Mumford stack with finite inertia, let
$Y$ be an algebraic stack with
separated diagonal, and let $f\colon X\to Y$ be a morphism.
Then the largest open substack $U$ of $X$ on which the
restriction of $f$ is a representable morphism enjoys the following
characterization: the geometric points of $U$ are precisely those
at which $f$ induces
an injective homomorphism of stabilizer group schemes.
\end{lemma}

\begin{proof}
Let $V\to Y$ be a smooth atlas, where $V$ is a separated scheme.
The hypotheses imply that $V\to Y$ is a separated morphism.
By Lemma \ref{lem:finiteinertia}, $X\times_YV$ has finite inertia.
Similarly, letting $S=V\times_YV$ (so that $S$ is isomorphic
to an algebraic space, with
$(\mathrm{pr}_1,\mathrm{pr}_2)\colon S\rightrightarrows V$
a groupoid presentation of $Y$),
$X\times_YS$ has finite inertia.
For a Deligne-Mumford stack with finite inertia, the
largest representable open substack is the complement of the
image of the complement of the identity in the inertia stack, and its
geometric points are characterized as those having trivial stabilizer group.
The largest representable open substack of $X\times_YV$
has the same pre-image by the maps
$\mathrm{id}_X\times \mathrm{pr}_i\colon
X\times_YS\to X\times_YV$
for $i=1$, $2$,
namely the
largest representable open substack of $X\times_YS$,
and hence determines an open substack $U$ of $X$.
It is easily verified that $U$ is the largest open substack of
$X$ on which the restriction of $f$ is representable, and that
the geometric points of $U$ are precisely those at which $f$
induces an injective homomorphism of stabilizer group schemes.
\end{proof}

\subsection{Local structure of Deligne-Mumford stacks}
\label{subsec:localDM}
The following result shows that Conjecture \ref{conj1} holds
for any Deligne-Mumford stack.

\begin{lemma}
\label{lem:DM}
Given a Deligne-Mumford stack $X$
with separated diagonal and a point $x$ of $X$ having finite stabilizer:
\begin{itemize}
\item[$\mathrm{(i)}$]
There exist an affine scheme $W$, finite group $G$,
action of $G$ on $W$, \'etale representable morphism
$f\colon [W/G]\to X$, and point $y\in [W/G]$ such that
$f(y)=x$ and $f$ is stabilizer preserving at $y$.
\item[$\mathrm{(ii)}$] The group $G$ in $\mathrm{(i)}$ may be taken to be the
geometric stabilizer group at $x$.
\end{itemize}
\end{lemma}

\begin{proof}
We obtain (i) from
\cite[Prop.\ 6.11]{rydh-quotients} and \cite[Thm.\ 6.1]{lmb}.
For (ii),
a variation of the argument of \cite[Prop.\ 6.11]{rydh-quotients} in which
$\mathscr{W}_d$ ($d\in \N$) is replaced by $\mathscr{W}_G$ ($G$ a finite group),
defined as the stack over $X$ of subschemes of the
pullbacks of the given \'etale atlas $U$,
\emph{equipped with a structure of $G$-torsor},
yields $V_G\to \mathscr{W}_G$ with $\mathscr{W}_G\cong [V_G/G]$.
\end{proof}

\begin{remark}
By \cite[Prop. 3.6 and Thm. 2.19]{aov}, Conjecture \ref{conj1} holds for an algebraic stack
with finite inertia.  By applying \cite[\S 4]{keel-mori}, we see that Conjecture \ref{conj1}
in fact holds for any algebraic stack with quasi-finite and separated diagonal.
\end{remark}

\subsection{Local structure of quotient stacks}
\label{subsec:slicing}
Let $k$ be an algebraically closed field.
Given a quotient stack $[U/G]$, with $U$ affine of finite type over $k$
and point $x\in U(k)$ with linearly reductive stabilizer group scheme $G_x$,
Luna's \'etale slice theorem \cite{luna} gives rise (under suitable hypotheses)
to an \'etale morphism $[W/G_x]\to [U/G]$, where
$W$ is a $G_x$-invariant
locally closed affine subscheme of $U$ containing $x$.

\begin{lemma}
\label{lem:affineslice}
Let $U$ be an affine scheme of finite type over an algebraically closed field
$k$, with action of a smooth linear algebraic group $G$.
Then, for every point $x\in U(k)$ with linearly reductive stabilizer
group scheme $G_x$ there exists a $G_x$-invariant
locally closed affine subscheme $W\subset U$
containing $x$ such that the induced morphism
$[W/G_x]\to [U/G]$ is \'etale.
In particular, Conjecture \ref{conj1} holds for $[U/G]$.
\end{lemma}

\begin{proof}
There exists a finite-dimensional linear $G$-space $V$ with
equivariant closed immersion $U\hookrightarrow V$; thus it suffices to
consider the case $U=V$.
By \cite[Lem.\ p.\ 96]{luna}, there exists a $G_x$-equivariant morphism
\[ g\colon V\to T_xV \]
which is \'etale at $x$ and satisfies $g(x)=0$.
We write
\[ T_xV=T_x(G\cdot x)\oplus N \]
for a $G_x$-representation $N$.
Then the representable morphism
$[g^{-1}(N)/G_x]\to [V/G]$
is \'etale at $x$, hence on $[W_0/G_x]$ for some
$G_x$-invariant open $W_0\subset g^{-1}(N)$ containing $x$.
Now $W$ may be taken to be any $G_x$-invariant affine neighborhood of $x$
in $W_0$.
\end{proof}

\begin{remark}
The conclusion of Lemma \ref{lem:affineslice}
is also valid if the hypothesis that $U$ is affine is replaced by the
hypothesis that $U$ is normal.
Let $x\in U(k)$ have linearly reductive stabilizer group scheme $G_x$.
Letting $G^\circ\subset G$ denote the connected component of the identity,
there exists a separated $G^\circ$-invariant open neighborhood of $x$
(e.g., the image under the action of $G^\circ\times U_0$ where
$U_0$ is any affine open neighborhood of $x$).
By Sumihiro's theorem \cite[Thm.\ 3.8]{sumihiro2}, there exists
a quasi-projective $G^\circ$-invariant neighborhood $U'$ of $x$.
Letting $H$ denote the image of $G_x$ in $G/G^\circ$ and
$G'$ the pre-image of $H$ under $G\to G/G^\circ$,
if we take $h_1$, $\dots$, $h_n\in G_x$ to be elements mapping onto $H$ then
$h_1U'\cap\dots\cap h_nU'$ is
a $G'$-invariant quasi-projective neighborhood of $x$.
So we are reduced to the case that $U$ is normal and quasi-projective.
Then there is an equivariant immersion $U\hookrightarrow \PP(V)$ for some
finite-dimensional linear $G$-space $V$, and as in the proof of
Lemma \ref{lem:affineslice} we are further reduced to the case $U=\PP(V)$.
We conclude by taking $f$ to be a $G_x$-semi-invariant homogeneous polynomial
with $f(x)\ne 0$,
invoking \cite[Lem.\ p.\ 96]{luna} to obtain
$G_x$-equivariant
\[ g\colon \PP(V)_f\to T_x\PP(V),\]
\'etale at $x$ with $g(x)=0$,
and applying the rest of the proof of Lemma \ref{lem:affineslice}.
\end{remark}

\begin{lemma} \label{lem:descent}
Let $Y$ be an algebraic stack locally of finite type
over an algebraically closed field $k$.  If there is a finite, flat cover
$f \colon X=[U / \GL_n] \to Y$ for some $n$,
where $U$ is an algebraic space (resp., an affine scheme), 
then $Y \cong [V / \GL_{n'}]$ for some $n'$, where
$V$ is an algebraic space (resp., an affine scheme).
\end{lemma}
 
 \begin{proof}
If $E$ is a
vector bundle on $X$  such that the stabilizer at every geometric point
acts faithfully on the fiber, then $f_*E$ is a vector bundle on $Y$ with
the same property (cf.\ \cite[proof of Lem.\ 2.13]{ehkv}).
If $U$ is affine, the base change $X
\times_Y V$ is also an affine scheme---indeed, as 
$U \to X$ is a $\GL_n$-bundle and 
$X \times_Y V \to X$ is a $\GL_{n'}$-bundle, 
the base change $U \times_Y V \to X \times_Y V$
is a $\GL_n$-bundle and since $U \times_Y V$ is affine, 
so is $X \times_Y V$.
Since $X \times_Y V \to V$ is finite and surjective, 
$V$ is an affine scheme.
\end{proof}

\begin{remark}
In \cite[\S 2]{rydh-noetherian}, Rydh defines an algebraic stack $X$ of finite 
presentation over $k$ to be \emph{of global type} if 
every point $x \in X(k)$ is in the image 
of an \'etale, representable morphism $[W/\GL_n] \to X$ where 
$W$ is quasi-affine.  Any algebraic stack satisfying
Conjecture \ref{conj1} which is also quasi-compact and has linearly reductive stabilizers at closed points
is therefore of global type.
\end{remark}

\subsection{Second reduction}
\label{subsec:secondreduction}
We give a result, reducing Conjecture \ref{conj1} for
$\fM_{g}^{\ss}$ to exhibiting certain finite covers
of \'etale neighborhoods.

\begin{prop}
\label{prop:finiteetale}
Let $X$ be an algebraic stack, locally of finite type over an
algebraically closed field $k$ with
separated and quasi-compact diagonal such that $X$ has affine stabilizer
groups at all closed points.
Suppose $x\in X(k)$ has linearly reductive stabilizer group scheme $G_x$,
and there exist morphisms
\[ X''\to X'\to X \]
and $x'\in X'(k)$ such that
\begin{itemize}
\item[$(\mathrm{i})$] $X''\cong [U/\GL_n]$ for some affine scheme $U$ over $k$
with action of $\GL_n$ for some $n$,
\item[$(\mathrm{ii})$] $X''\to X'$ is a finite flat cover,
\item[$(\mathrm{iii})$] $X'\to X$ is \'etale representable,
\item[$(\mathrm{iv})$] $X'\to X$ is stabilizer preserving at $x'$.
\end{itemize}
Then Conjecture \ref{conj1} holds for $X$ at the point $x$.
\end{prop}

\begin{proof}
By Lemma \ref{lem:descent} we have $X'\cong [V/GL_{n'}]$.
We conclude by Lemma \ref{lem:affineslice}.
\end{proof}

\section{Proof in a special case} \label{sec:specialcase}
We apply the above reduction steps
(Propositions \ref{prop:first} and \ref{prop:finiteetale}) and a structural result
(Lemma \ref{lem:localstr}, below) to give a proof of
Theorem \ref{main-theorem} in a special case.

\subsection{Aligned log structure}
\label{subsec:aligned}
We recall the definition of an
aligned log structure introduced in \cite{acfw}.
Let $X$ be a scheme.
Given a log structure $M\to \mathcal{O}_X$ we consider the
characteristic sheaf $\overline{M}=M/\mathcal{O}_X^*$.
If $M$ is a locally free log structure, then
there is a 
subsheaf \emph{of sets}
$\overline{M}^1\subset \overline{M}$,
whose values on stalks are
the sums of subsets of generators of the free monoids
$\overline{M}_{\bar x}$ at geometric points $\bar x\in X$.
An \emph{aligned log structure} is a locally free log structure
together with a subsheaf $\mathcal{S} \subset \overline{M}^1$, such that for every geometric point
$\bar x\in X$ there is a labeling of the generators of
$\overline{M}_{\bar x}$ as $e_1, \dots, e_n$, such that
\[\mathcal{S}_{\bar x} = \{0,e_1,e_1+e_2,\ldots,e_1+\cdots+e_n\}.\]

\begin{prop} \label{prop:alignedlog}
Let $E$ be a Deligne-Mumford stack, $E\to \bar{\cM}_g$ an \'etale morphism with
corresponding family of curves $C\to E$, and $D\subset C^{\sing}$ an
open and closed subscheme which maps isomorphically to a divisor in $E$.
Define
\[ X:=E\times_{\bar{\cM}_g} \fM_g^{\ss}, \]
and let $\fC \to X$ denote the associated family of semistable curves,
with stabilization $\fC\to C\times_EX$.
\begin{itemize}
\item[$\mathrm{(i)}$]
The stack
\[ D\times_{C^{\sing}} \fC^{\sing} \]
is, by stabilization, the normalization of $D\times_EX$.
\item[$\mathrm{(ii)}$] Suppose the family of curves $D\times_EC\to D$
is obtained by gluing some $C_0\to D$ along two sections $s_1$ and $s_2$,
as in \S \ref{subsec:boundary-Mg}.
Then there is a unique aligned log structure $(M,S)$ on $X$ whose underlying
log structure is that of the normal crossing divisor $D\times_EX$,
such that for every geometric point $\bar x$ of $D\times_EX$, corresponding to a
prestable curve whose stabilization collapses
$r$ exceptional components to the corresponding point of $D$,
the elements of $S_{\bar x}$ under the isomorphism of
$\mathrm{(i)}$ are $\bar x_1+\dots+\bar x_j$ for $0\le j\le r+1$, where
$\bar x_1$, $\dots$, $\bar x_{r+1}$ are the points of
$\fC^{\sing}$ mapping to the point of $D$,
ordered so that $\bar x_1$ is the image of $s_1$ under
the rational map
$(C_0)_{\bar x}\dashrightarrow \fC_{\bar x}$, and for every $i$
the points $\bar x_i$ and $\bar x_{i+1}$ lie on one of the exceptional components
collapsing to the point of $D$ corresponding to $\bar x$.
\end{itemize}
\end{prop}

\begin{proof}
Since $D\times_{C^{\sing}} \fC^{\sing}$ is smooth and
the morphism to $D\times_EX$ is finite and restricts to an isomorphism over
the stable curves in $X$, assertion $\mathrm{(i)}$ is clear.

Assertion $\mathrm{(ii)}$ quickly reduces to the following claim.
Let $R$ be a henselian discrete valuation ring with residue field $k$,
together with $C_{0,R}\to \Spec(R)$ with two sections and $\fC_R\to \Spec(R)$,
corresponding to some $\Spec(R)\to D\times_EX$; in particular, the
stabilization of $\fC_R$ is assumed to be identified with the stable curve
$C_R$ obtained by gluing from $C_{0,R}\to \Spec(R)$ with the two sections.
Denoting by $s_{1,R}$ and $s_{2,R}$ the two sections, the common composite
with the gluing is a section $\Spec(R)\to C_R$ with image
contained in $C_R^{\sing}$.
The restriction
\[ h\colon \Spec(R)\times_{C_R^{\sing}}\fC_R^{\sing}\to \Spec(R) \]
of the finite unramified morphism of $\mathrm{(i)}$ is necessarily the
projection of a disjoint union of some copies of $\Spec(R)$ and some
(possibly nonreduced) points mapping to the closed point of
$\Spec(R)$.
Let the fiber over the closed point be $x_1$, $\dots$, $x_{r+1}$ as in
$\mathrm{(ii)}$, and
let
\[ i_1<i_2<\dots<i_{q+1} \]
be the indices of those points over the closed point, which are the
specializations of copies of $\Spec(R)$.
The claim is that at a geometric general point $\bar\eta$ of $\Spec(R)$,
stabilization of $\fC_{\bar\eta}$ collapses a chain of $q$ rational curves
to the node of $C_{\bar\eta}$ (the common image of the two sections under gluing),
and with the points $\bar\eta_1$, $\dots$, $\bar\eta_{q+1}$ as in
$\mathrm{(ii)}$ the point $\bar\eta_a$ maps to the $i_a$-copy of $\Spec(R)$
for $a=1$, $\dots$, $q+1$.

There is nothing to prove if $q=0$, so we assume $q\ge 1$.
There is an iterated blow-up $\tau\colon \widehat{C}_{0,R}\to C_{0,R}$
at points over $c_1:=s_{1,R}(\Spec(k))$ and $c_2:=s_{2,R}(\Spec(k))$ such that the
composite with gluing and the inverse of stabilization
\[ \widehat{C}_{0,R}\stackrel{\tau}\to C_{0,R}\to C_R\dashrightarrow \fC_R \]
is defined in a neighborhood of
$\tau^{-1}(\{c_1,c_2\})$.
(The process of blowing up points of indeterminacy terminates.
A classical result of Northcott asserts that upon blowing up and replacing an
$\mathfrak{m}_{c_i}$-primary ideal in $\mathcal{O}_{C_{0,R},c_i}$
by its proper transform the multiplicity decreases; cf.\ \cite{huneke}.)
Now $\tau^{-1}(c_1)$ is a connected scheme, whose image $Z_0\subset \fC_R$
contains $x_1$ and the specialization of $\eta_1$.

For each $1\le a\le q$ the generic fiber $\fC_\eta$ has a unique
exceptional component $E_{\eta,a}$ containing $\eta_a$ and $\eta_{a+1}$.
We let $W_a$ denote the closure of $E_{\eta,a}$ in $\fC_R$.
The special fiber $Z_a$ of $W_a$ is connected and contains the
specializations of $\eta_a$ and $\eta_{a+1}$.

Since $\fC_R$ is Cohen-Macaulay with reduced special fiber,
the complement of the closure of $\fC_\eta^{\sing}$ is
normal.
In particular, each $x_i$ with $i\notin \{i_1,\dots,i_{q+1}\}$ is a
normal point of $\fC_R$, as are all smooth points of the
special fiber of $\fC_R$.

With these observations, we may establish the claim by contradiction.
Suppose there is some $a$, which we take to be minimal,
such that $\eta_a$ specializes to $x_i$ and
$\eta_{a+1}$ specializes to $x_j$, with $j<i$.
If $a=1$ then since $Z_0$ is connected and contains $x_1$ and $x_i$
we have $x_j\in Z_0$ as well, and then since
$Z_1$ is connected and $1$-dimensional, we would have $\dim(Z_0\cap Z_1)=1$,
which is a contradiction.
If $a>1$, we may argue similarly, using that $Z_0\cup\dots\cup Z_{a-1}$ is
connected and contains $x_1$ and $x_i$, hence must as well contain $x_j$.
\end{proof}

\subsection{Separating nodes with distinct genera}
\label{subsec:separating}
From \S \ref{subsec:boundary-Mg}, every partition of $g$ as a sum of two
positive integers determines an irreducible component of the boundary divisor
of $\bar{\cM}_g$.
Suppose $g_1+g_2=g$ with $g_1>g_2$.
Then the corresponding boundary component has an open substack
$D_{g_1,g_2}$ parametrizing unions of smooth curves of genera $g_1$ and $g_2$
at a node.
We define
\[ E_{g_1,g_2}:=\cM_g\cup D_{g_1,g_2}\subset \bar{\cM}_g, \]
the open substack consisting of smooth curves of genus $g$ and unions
of smooth curves of genera $g_1$ and $g_2$.
Then the inclusion of $E_{g_1,g_2}$ in $\bar{\cM}_g$ and
$D_{g_1,g_2}$ satisfy the hypotheses of Proposition \ref{prop:alignedlog},
hence determine an aligned log structure on
\[ X_{g_1,g_2}:=E_{g_1,g_2}\times_{\bar{\cM}_g} \fM_g^{\ss}, \]
i.e., a morphism
\begin{equation}
\label{allogstr}
X_{g_1,g_2}\to \mathcal{L}og^{\mathrm{al}}
\end{equation}
to the stack of aligned log structures; the stack
of aligned log structures is described in \cite{acfw}.

\begin{thm}
\label{main-theorem-special-case}
For positive integers $g_1$, $g_2$, and $g$, with $g_1+g_2=g$ and $g_1>g_2$,
Conjecture \ref{conj1} holds for $X_{g_1,g_2}$.
\end{thm}

Key to the proof is the following observation, which reveals the
structure of the moduli stack of semistable curves of genus $g$ near a curve
whose stabilization is a union of smooth curves of genera $g_1$ and $g_2$.
The quotient stack
\[ \mathcal{A}^1:=[\A^1/\G_m], \]
where $\G_m$ acts in the standard way on $\A^1$,
plays an important role.
As is recalled in \cite{acfw},
the stack $\mathcal{A}^1$ serves as a universal target for pairs
consisting of a smooth variety or algebraic stack $Y$ with divisor $D\subset Y$.
If $D$ is also smooth, then the corresponding morphism $Y\to \mathcal{A}^1$
is smooth.

\begin{lemma}
\label{lem:localstr}
The projection from $X_{g_1,g_2}$ and morphism \eqref{allogstr} fit into a
fiber square
\[
\xymatrix{
X_{g_1,g_2}\ar[r]\ar[d]  & E_{g_1,g_2}\ar[d] \\
\mathcal{L}og^{\mathrm{al}} \ar[r] & \mathcal{A}^1
}
\]
where the morphisms to $\mathcal{A}^1$ correspond to the boundary divisor
of $\mathcal{L}og^{\mathrm{al}}$ and the divisor
$D_{g_1,g_2}\subset E_{g_1,g_2}$.
\end{lemma}

\begin{proof}
The horizontal morphisms restrict to isomorphisms on the complements
of substacks of codimension $2$, hence so does the morphism
\begin{equation}
\label{eqn.Xg1g2mor}
X_{g_1,g_2}\to \mathcal{L}og^{\mathrm{al}}\times_{\mathcal{A}^1}E_{g_1,g_2},
\end{equation}
which is representable.
The fiber product in \eqref{eqn.Xg1g2mor} is smooth, since
the morphism $E_{g_1,g_2}\to [\A^1/\G_m]$ is smooth.
So, by Zariski-Nagata purity \cite[Thm.\ X.3.1]{sga1} the morphism
\eqref{eqn.Xg1g2mor} is \'etale.
The morphism \eqref{eqn.Xg1g2mor} is bijective on geometric points
and pointwise stabilizer preserving,
hence by Lemma \ref{lem:etale} is an isomorphism.
\end{proof}

For the proof of Theorem \ref{main-theorem-special-case} we use the stacks
\[ \mathcal{A}^n:=[\A^n/\G_m^n], \]
also introduced in \cite{acfw}
(quotient stacks for the standard group actions).

\begin{proof}[Proof of Theorem \ref{main-theorem-special-case}]
Let $x\in X_{g_1,g_2}(k)$ correspond to a semistable curve $C$ of
genus $g$ whose stabilization $C'$ is the union of smooth curves of
genera $g_1$ and $g_2$.
We verify Conjecture \ref{conj1} for $X_{g_1,g_2}$ at $x$.
Let $H$ be the automorphism group of $C'$.
By Lemma \ref{lem:DM}, there exist an affine scheme $W$ with action of $H$
fixing a point $y\in W$
and an \'etale representable morphism
\[ [W/H]\to E_{g_1,g_2} \]
sending $y$ to the point of $E_{g_1,g_2}$ corresponding to $C'$.

If $C$ is stable, we are done, so we assume that some
positive number $r$ of exceptional components are collapsed by
stabilization $C'\to C$.
The coordinate hyperplanes gives rise to an aligned log structure on
$\mathcal{A}^{r+1}$, and the corresponding morphism
$\mathcal{A}^{r+1}\to \mathcal{L}og^{\mathrm{al}}$ is
\'etale representable \cite{acfw} and stabilizer preserving at the origin.
Consequently, we have a stack
\[ X':=\mathcal{A}^{r+1}\times_{\mathcal{L}og^{\mathrm{al}}}X_{g_1,g_2}
\times_{E_{g_1,g_2}}[W/H], \]
such that the projection to $X_{g_1,g_2}$ is \'etale representable,
and a point $x'\in X'$ at which the
projection to $X_{g_1,g_2}$ is stabilizer preserving, sending $x'$ to $x$.

We let
\[ X'':=\mathcal{A}^{r+1}\times_{\mathcal{L}og^{\mathrm{al}}}X_{g_1,g_2}
\times_{E_{g_1,g_2}}W. \]
By Lemma \ref{lem:localstr}, $X''\cong \mathcal{A}^{r+1}\times_{\mathcal{A}^1}W$.
Since $W\to \mathcal{A}^1$ is affine,
$X''$ is of the form $[U/\G_m^{r+1}]$ with $U$ affine,
and Proposition \ref{prop:finiteetale} implies
Conjecture \ref{conj1} for $X_{g_1,g_2}$ at $x$.
\end{proof}

\section{Stacks generalities} \label{sec:stack}
In this section, we record general facts about algebraic stacks that will be useful for the proof of Theorem \ref{main-theorem}.  The reader may want to skip this section on the first reading but later return when these results are applied in \S \ref{sec:local}-\ref{sec:proof}.

\subsection{\'Etale coverings from unramified morphisms}
\label{subsec:etalecoverings}
Rydh \cite[\S 3]{rydh-can} associates to a representable unramified morphism
$X\to Y$ of algebraic stacks a stack $E_{X/Y}$ and factorization
$X\to E_{X/Y}\to Y$.\footnote{In fact, Rydh's construction works more generally for 
non-representable morphisms.}
Musta\c t\u a and Musta\c t\u a \cite[Prop.\ 1.2, Thm.\ 1.5]{mm}
associate to a representable unramified morphism $X\to Y$,
\'etale on its image $Y_1\subset Y$, a stack $F_{X/Y}$ and factorization
$X\to F_{X/Y}\to Y$.
The morphism to $Y$ from $E_{X/Y}$, respectively from $F_{X/Y}$, is
\'etale representable, with the objects over a scheme $T$, are defined to be morphisms
$T\to Y$ (respectively subject to the requirement that the 
image is disjoint from $\bar{Y}_1 \smallsetminus Y_1$), 
together with pairs consisting of
a closed subscheme $T'\subset T$ and an open immersion
$T'\to X\times_YT$ over $T$ (respectively subject to the
requirement $T'=Y_1\times_YT$).
In these constructions, the morphisms $X \to E_{X/Y}$ and $X \to F_{X/Y}$ are closed
immersions.

\begin{prop}
\label{prop:etalestuff}
Let $f\colon X\to Y$ be an \'etale representable morphism of
algebraic stacks. For a locally closed substack $W\subset Y$ with
$f(X)\cap \overline{W}\subset W$ and fiber diagram
\[\xymatrix{
Z\ar[r] \ar[d] & X \ar[d]^f \\
W \ar[r] & Y
}\]
the induced morphism $g\colon X\to F_{Z/Y}$ is an isomorphism if and only if
the restriction of $f$
over $Y\smallsetminus \overline{W}$ is an isomorphism.
\end{prop}

\begin{proof}
The forward implication is clear from the definition
of $F_{Z/Y}$, and the
reverse implication follows by applying Lemma \ref{lem:etale} to $g$
by noting that the geometric points
both of $X$ and of $F_{Z/Y}$ lie over
the union of $W$ and $Y\smallsetminus \overline{W}$ and observing that
$g$ restricts to isomorphisms
over $W$ and over $Y\smallsetminus \overline{W}$.
\end{proof}

\subsection{Some stacks with non-separated diagonal}
\label{subsec:nonsepdiag}
The stack of log structures \cite{olssonloggeometry} has
quasi-compact but non-separated
diagonal.
Here we describe an algebraic stack $\mathring{\cA}^r$ for each $r > 0$,
also with quasi-compact but non-separated diagonal.

We recall, from \S \ref{subsec:separating}, the stack
$\mathcal{A}^r=[\A^r/\G_m^r]$
(quotient stack for the standard action of $\G_m^r$ on $\A^r$).
There is an involution $\iota\colon \A^r\to \A^r$,
$(x_1,\ldots,x_r)\mapsto (x_r,\ldots,x_1)$ and a similar compatible
involution of $\G_m^r$ inducing an involution
$\iota\colon \mathcal{A}^r\to \mathcal{A}^r$.

Let us introduce $\mathring{\mathcal{A}}^r$ by defining
$U_r$ to be the non-separated scheme which is the union of two copies of
$\A^r\times \G_m^r$ along the two copies of
$(\A^1\smallsetminus\{0\})^r\times \G_m^r$ identified via the involution
\[(x_1,\ldots,x_r,t_1,\ldots,t_r)\mapsto
(x_1,\ldots,x_r,t_r\frac{x_r}{x_1},\ldots,t_1\frac{x_1}{x_r}),\]
and setting
\[\mathring{\mathcal{A}}^r=[U_r\rightrightarrows \A^r],\]
the stack associated to the groupoid scheme with
projection map and twisted action map which
on the first copy of $\A^r\times \G_m^r$ is the standard diagonal action
and on the second copy is the composition of the involution $\iota$ of
$\A^r$ with the standard diagonal action.

The stabilizer of $\mathring{\cA}^r$ at the origin is the semi-direct product
$\Gm^{r} \rtimes \ZZ/2 \ZZ$ where $\ZZ / 2\ZZ$ acts on $\Gm^r$ via 
the involution $(t_1, \ldots, t_r) \mapsto (t_r, \ldots, t_1)$ of $\Gm^r$.

Just as the stack $\mathcal{A}^1$ has the well-known interpretation of
schemes with a line bundle and a global section, we leave it to the
reader to see that $\mathring{\mathcal{A}}^1$ has the interpretation of
schemes with a line bundle, a global section, and a degree-two finite
\'etale cover of the zero-locus of the section.
In particular, the fiber of the obvious morphism
$\mathring{\mathcal{A}}^1\to\mathcal{A}^1$ over
$B\G_m=[\{0\}/\G_m]$ is isomorphic to $B\G_m\times B\Z/2\Z$;
Proposition \ref{prop:etalestuff} yields an isomorphism
\begin{equation}
\label{eqn:A1ringF}
F_{B\G_m/\mathring{\mathcal{A}}^1}\cong \mathcal{A}^1.
\end{equation}

Multiplication of coordinates induces morphisms 
$\mathrm{mul} \colon \cA^r \to \cA^1$ and 
$\mathrm{mul}\colon \mathring{\cA}^r \to \mathring{\cA}^1$.  
There are \'etale representable morphisms 
$\cA^r \to \mathring{\cA}^r$ induced on the level of
groupoid schemes from the inclusion of the
first copy of $\A^r\times\G_m^r$ in $U_r$.

\begin{prop}
\label{prop:Acirclemult}
These morphisms fit into a fiber diagram
\[
\xymatrix{
\mathcal{A}^r\ar[r]^{\mathrm{mul}}\ar[d] &
\mathcal{A}^1\ar[d] \\
\mathring{\mathcal{A}}^r\ar[r]^{\mathrm{mul}} &
\mathring{\mathcal{A}}^1
}
\]
\end{prop}

\begin{proof}
Each of the morphisms arises from a morphism of groupoid schemes.
Standard manipulations of groupoid schemes establish the proposition.
\end{proof}


\subsection{Semi-aligned log structures}
\label{subsec:semialigned}
We return to the setting of \S \ref{subsec:aligned}.
Besides the zero section there is another global section $s_1$ of
$\overline{M}^1$, given by the sums of generators of the
$\overline{M}_{\bar x}$.
An involution $j$ of $\overline{M}^1$ is characterized by the property
that the composite
\[
\overline{M}^1\stackrel{(\mathrm{id},j)}\longrightarrow
\overline{M}^1\times \overline{M}^1\subset
\overline{M}\times \overline{M}\stackrel{+}\longrightarrow
\overline{M}
\]
is the constant map $s_1$.
Given an aligned log structure, application of the involution $j$
yields a new aligned log structure,
the \emph{opposite aligned log structure}.

A related notion is \emph{semi-aligned log structure}, exactly as in \S \ref{subsec:aligned} but
with
\[\mathcal{S}_{\bar x} = \{0,e_1,e_n,e_1+e_2,e_{n-1}+e_n,\ldots,e_1+\cdots+e_n\}.\]

\begin{prop}
\label{prop:traceidempotentmonoids}
Let $\pi\colon X\to Y$ be a finite type \'etale universally closed morphism
of algebraic spaces.
Then there is a natural transformation
\[
tr\colon \pi_*\circ \pi^*\to \mathrm{id}
\]
(``trace'')
on sheaves of idempotent monoids on $Y$, which on stalks is given by
summing the values on fibers.
\end{prop}

\begin{proof}
Given $y\in Y$ there exist an \'etale neighborhood
$y'\in Y'\to Y$ with $y'\mapsto y$ and sections
$\sigma_1, \dots, \sigma_r$ of $\pi'\colon X':=X\times_YY'\to Y'$
whose images cover $\pi'^{-1}(y')$.
Since $\pi$ is universally closed, we may replace $Y'$ by a suitable open
neighborhood of $y'$, so that the images of the sections cover $X'$.
We may then define the trace locally by summing the $r$ sections over $Y'$.
\end{proof}

Besides the stack of aligned log structures
$\mathcal{L}og^{\mathrm{al}}$ described in \cite{acfw},
there is a stack of semi-aligned log structures, which we denote
by $\mathcal{L}og^{\frac{1}{2}\mathrm{al}}$.
The assignment, to a scheme with an aligned log structure, of the opposite
aligned log structure, yields an involution $j$ of the stack
$\mathcal{L}og^{\mathrm{al}}$.
If $X$ is a scheme and
$(M\to \mathcal{O}_X, \mathcal{S}\subset \overline{M}^1)$
is an aligned log structure on $X$, then
$\mathcal{S} \sqcup j(\mathcal{S})$ is an
aligned log structure on $X\sqcup X$, and the trace for
$X\sqcup X\to X$ yields
$tr(\mathcal{S} \sqcup j(\mathcal{S}))$, a
semi-aligned log structure on $X$.
This construction gives rise to a morphism of algebraic stacks
\begin{equation}
\label{eqn:alignedtosemialigned}
\mathcal{L}og^{\mathrm{al}}\to \mathcal{L}og^{\frac{1}{2}\mathrm{al}}.
\end{equation}

\begin{prop}
\label{prop:logstructures}
The morphism \eqref{eqn:alignedtosemialigned} is
representable, \'etale, of finite type, and
universally closed,  and restricts to an
isomorphism over the locus of locally free log structures of
rank $\le 1$ and  a finite \'etale morphism of degree $2$ over
the locus of rank $\ge 2$.
\end{prop}

\begin{proof}
The morphism \eqref{eqn:alignedtosemialigned} is clearly representable and
locally of finite type,
and it is trivial to verify the
criterion in terms of square-zero extensions to be \'etale.
The sheaf of sections of \eqref{eqn:alignedtosemialigned}
is locally the quotient
of a finite constant sheaf, hence the morphism \eqref{eqn:alignedtosemialigned}
is universally closed of finite type.
The assertion about the locus of locally free log structures of
rank $\le 1$ is a triviality.
Since the geometric fibers over locally free log structures of rank $\ge 2$
all consist of two points, the properties of the morphism \eqref{eqn:alignedtosemialigned}
 imply that it restricts to a finite morphism.
 \end{proof}

Given a normal crossings divisor $D$ on a smooth algebraic stack $X$, 
we set $X^{k} \subset X$ to be the locally closed 
substack consisting of all points lying, smooth locally in 
exactly $k$ smooth divisors.  
For instance, 
$X^0 = X \setminus D$ 
and $X^1$ is the smooth locus of $D$.  
Set $X^{\le k}$ and $X^{\ge k}$ to be the 
locally closed substacks obtained by taking the unions
of the corresponding  $X^{i}$'s.
We make these definitions more generally for an algebraic stack with a given
locally free log structure.

Considering the fiber square
\begin{equation}
\begin{split}
\label{eqn:squareforcorlogstructures}
\xymatrix{
\mathcal{L}og^{\mathrm{al},\ge 2} \ar[r] \ar[d] &
\mathcal{L}og^{\mathrm{al}} \ar[d] \\
\mathcal{L}og^{\frac{1}{2}\mathrm{al},\ge 2} \ar[r] &
\mathcal{L}og^{\frac{1}{2}\mathrm{al}}
}
\end{split}
\end{equation}
Proposition \ref{prop:logstructures} tells us that the criterion of
Proposition \ref{prop:etalestuff} is satisfied.

\begin{cor}
\label{cor:logstructures}
The fiber square \eqref{eqn:squareforcorlogstructures} induces
an isomorphism
\[
\mathcal{L}og^{\mathrm{al}}\to
F_{\mathcal{L}og^{\mathrm{al},\ge 2}/\mathcal{L}og^{\frac{1}{2}\mathrm{al}}}
\] 
\end{cor}

Another consequence of Proposition \ref{prop:logstructures} is that
we have a morphism
\[
\mathrm{al}\colon \mathcal{L}og^{\frac{1}{2}\mathrm{al},\ge 2}\to B\Z/2\Z
\]
corresponding to the restriction of morphism \eqref{eqn:alignedtosemialigned}
to $\mathcal{L}og^{\frac{1}{2}\mathrm{al},\ge 2}$.
Now we consider the composite morphism
\begin{equation}
\label{eqn:tomakeuglyF}
\mathcal{L}og^{\frac{1}{2}\mathrm{al},\ge 2}
\stackrel{(\mathrm{id},\mathrm{al})}\longrightarrow
\mathcal{L}og^{\frac{1}{2}\mathrm{al},\ge 2}\times B\Z/2\Z
\hookrightarrow
\mathcal{L}og^{\frac{1}{2}\mathrm{al}}\times_{\mathcal{A}^1}
\mathring{\mathcal{A}}^1,
\end{equation}
remembering that the fiber of
$\mathring{\mathcal{A}}^1$ over $B\G_m\subset \mathcal{A}^1$
is isomorphic to $B\G_m\times B\Z/2\Z$,
so that we obtain a stack
\begin{equation}
\label{eqn:uglyF}
F_{\mathcal{L}og^{\frac{1}{2}\mathrm{al},\ge 2}/
\mathcal{L}og^{\frac{1}{2}\mathrm{al}}\times_{\mathcal{A}^1}
\mathring{\mathcal{A}}^1}.
\end{equation}

We note that while $\mathcal{L}og^{\mathrm{al}, r}$ consists of a single point 
with stabilizer $\Gm^r$ and $\mathcal{L}og^{\frac{1}{2}\mathrm{al}, r}$ consists 
of a single point with
stabilizer $\Gm^r$ for $r \le 1 $ and $\Gm^r \rtimes \ZZ/2\ZZ$ for $ r \ge 2$, 
the stack \eqref{eqn:uglyF} in codimension $r$ is a point
with trivial stabilizer for $r =0$ and
stabilizer $\Gm^r \rtimes \ZZ/2\ZZ$ for $r \ge 1$.
This feature leads us to introduce a new notion, and
compact notation.

By an \emph{augmented semi-aligned log structure} on a scheme or
algebraic stack $X$ we mean the data of
a semi-aligned log structure $(M,S)$,
a degree $2$ finite \'etale cover $\widehat{X}\to X^{\ge 1}$, and an
\emph{augmentation isomorphism}
\[
\widehat{X}^{\ge 2}\to
\mathcal{L}og^{\mathrm{al}}
\times_{\mathcal{L}og^{\frac{1}{2}\mathrm{al}}}
X^{\ge 2}
\]
over $X^{\ge 2}$.
(The fiber product is of the morphism \eqref{eqn:alignedtosemialigned}
and the restriction of $(M,S)$ to $X^{\ge 2}$.)
The stack \eqref{eqn:uglyF} is identified with the
stack of augmented semi-aligned log structures, which we denote by
\[
\mathcal{L}og^{(\frac{1}{2}+\varepsilon)\mathrm{al}}.
\]
To understand the stack \eqref{eqn:uglyF}, following the definition
in \S \ref{subsec:etalecoverings},
it is important to remember
the isomorphism of target objects in
the definition of fiber product of stacks
$\mathcal{L}og^{\frac{1}{2}\mathrm{al},\ge 2}\times_{
\mathcal{L}og^{\frac{1}{2}\mathrm{al}}\times_{\mathcal{A}^1}
\mathring{\mathcal{A}}^1}T$
(target of the open immersion).

As an example, there is a natural semi-aligned log structure on
$\mathring{\mathcal{A}}^r$,
by taking the $e_1$, $\dots$, $e_n$ to correspond to the natural
coordinates on $\A^r$.
The triple consisting of this semi-aligned log structure,
the cover
$(\mathcal{A}^r)^{\ge 1}\to (\mathring{\mathcal{A}}^r)^{\ge 1}$,
and the restriction of the natural aligned log structure of
$\mathcal{A}^r$, is an augmented semi-aligned log structure.  This yields 
a morphism
\begin{equation} \label{E:augmentedA1}
\mathring{\mathcal{A}}^r\to\mathcal{L}og^{(\frac{1}{2}+\varepsilon)\mathrm{al}}.
\end{equation}

An aligned log structure determines an augmented semi-aligned
log structure (via the trace construction as in \eqref{eqn:alignedtosemialigned},
with trivial cover and isomorphism of trivial covers as
augmentation isomorphism).
In other words, we have a morphism from $\mathcal{L}og^{\mathrm{al}}$ to
$\mathcal{L}og^{(\frac{1}{2}+\varepsilon)\mathrm{al}}$, and
this is \'etale since the stack \eqref{eqn:uglyF} is \'etale over
$\mathcal{L}og^{\frac{1}{2}\mathrm{al}}$ and the composite
\[ \mathcal{L}og^{\mathrm{al}}\to
\mathcal{L}og^{(\frac{1}{2}+\varepsilon)\mathrm{al}}\to
\mathcal{L}og^{\frac{1}{2}\mathrm{al}} \]
is \'etale.  There is also a morphism
\begin{equation} \label{E:augmentedtoA1circ}
\mathcal{L}og^{(\frac{1}{2}+\varepsilon)\mathrm{al}}\to \mathring{\mathcal{A}}^1,
\end{equation}
mapping an augmented semi-aligned log structure to its degree $2$ cover.

\begin{prop}
\label{prop:Araugmentedsemialigned}
The morphisms \eqref{E:augmentedA1} and \eqref{E:augmentedtoA1circ}
fit into a fiber diagram
\[
\xymatrix{
\mathcal{A}^r \ar[r] \ar[d] & \mathcal{L}og^{\mathrm{al}} \ar[r] \ar[d] & \mathcal{A}^1 \ar[d]\\
\mathring{\mathcal{A}}^r \ar[r] & \mathcal{L}og^{(\frac{1}{2}+\varepsilon)\mathrm{al}} \ar[r] &
\mathring{\mathcal{A}}^1
}
\]
\end{prop}

\begin{proof}
Since the outer square is a fiber square
(Proposition \ref{prop:Acirclemult}), it suffices to verify that the
right-hand square is a fiber square.
This results by observing that upon base change by
$\mathcal{A}^1\to \mathring{\mathcal{A}}^1$ the composition \eqref{eqn:tomakeuglyF}
yields a composition of morphisms in \eqref{eqn:squareforcorlogstructures}
and applying Corollary \ref{cor:logstructures}.
\end{proof}

\begin{prop}
\label{prop:augmentedsemialigned}
Let $X$ be an algebraic stack with locally free log structure $M$,
and let $\pi\colon\widehat{X}\to X^{\ge 1}$ be a degree $2$ finite \'etale cover
with corresponding \'etale cover $f\colon F_{\widehat{X}/X}\to X$
and involution $\iota\colon F_{\widehat{X}/X}\to F_{\widehat{X}/X}$.
If $(f^*M,S')$ is an aligned structure with
$\iota^*S'$ related to $S'$ by the involution $j$,
then there is a unique augmented semi-aligned log structure
\[ ((M,S),\pi,\widehat{X}^{\ge 2}\to \mathcal{L}og^{\mathrm{al}}
\times_{\mathcal{L}og^{\frac{1}{2}\mathrm{al}}}
X^{\ge 2}) \]
on $X$ such that
\begin{itemize}
\item[$\mathrm{(i)}$]
the semi-aligned log structure on
$F_{\widehat{X}/X}$ obtained by the trace construction from $S'$ is
equal to $f^*S$, and
\item[$\mathrm{(ii)}$]
the restriction of $(f^*M,S')$ to $\widehat{X}^{\ge 2}$ 
is the aligned log structure obtained by projecting from the
augmentation isomorphism.
\end{itemize}
Furthermore, the morphisms corresponding to the log structures fit into
a fiber diagram
\[
\xymatrix{
F_{\widehat{X}/X} \ar[r] \ar[d] & \mathcal{L}og^{\mathrm{al}} \ar[d] \\
X \ar[r] & \mathcal{L}og^{(\frac{1}{2}+\varepsilon)\mathrm{al}}
}
\]
\end{prop}

\begin{proof}
Condition (i) determines $S$ uniquely, and $(M,S)$ is a
semi-aligned log structure on $X$.
Restricting $S'$, we obtain
$\widehat{X}\to \mathcal{L}og^{\mathrm{al}}$, and the further restriction
to $\widehat{X}^{\ge 2}$ yields the
augmentation isomorphism, which is uniquely determined by (ii).

The diagram is 2-commutative.
As in the proof of Proposition \ref{prop:Araugmentedsemialigned},
we may reduce the proposition,
by adjoining the fiber square with
$\mathcal{A}^1 \to \mathring{\mathcal{A}}^1$ to the right,
to the assertion that the outer square of the larger diagram is 2-cartesian.
Since the degree $2$ cover of the augmented semi-aligned log structure is $\pi$,
we have the assertion by the isomorphism \eqref{eqn:A1ringF} and the fact that the
construction of \S \ref{subsec:etalecoverings} respects base change.
\end{proof}

\section{A local construction around a strictly semistable curve} \label{sec:local}
In this section, we construct an \'etale neighborhood
$[(X_{g, \triangle}^{\times m})^{\circ} / S_{\br} ] \to \mathfrak{M}^{\ss}_g$
of any strictly semistable curve $C$ which is stabilizer preserving at a
pre-image $x'$ of $C$
(Proposition \ref{prop-map-to-mgss}).
In order to apply
Proposition \ref{prop:finiteetale} to $\fM_g^{\ss}$ (in Section \ref{sec:proof})
we will need a refinement (Proposition \ref{prop:second})
that leads to the existence of $U\to \mathfrak{M}^{\ss}_g$,
\'etale  and representable, and
stabilizer preserving at a pre-image $x''$ of $C$
(Corollary \ref{cor:second}).

The morphism
$D_g \to \bar{\cM}_g$
is representable and unramified,
and we may consider the algebraic stack
$E_g=E_{D_g/\bar{\cM}_g}$ (construction of \S \ref{subsec:etalecoverings}),
with \'etale representable morphism
\[
E_g \to \bar{\cM}_g.
\]
The stack $E_g$ may be viewed as the moduli stack
of stable curves of genus $g$ endowed with a 
choice of at most one node.

\subsection{Stabilization}
\label{subsec:stabilization}
We define $X_g$ to be the algebraic stack parametrizing 
semistable curves of genus $g$ with at most one chosen node of the stabilization, that is, the fiber product
\[ X_g = \fM_g^{\ss} \times_{\bar{\cM}_g} E_g. \]
We will, by abuse of notation, let $\st$ denote also the projection
$X_g\to E_g$.

Let $X_{g,\triangle} \subset X_g$ be the open substack parametrizing 
semistable curves $C$ whose stabilization $C \to C'$ has at most one
positive dimensional fiber, such 
that if $c' \in C'$ has a positive dimensional fiber then the node $c'$ is
chosen.

\subsection{The local construction} \label{subsec:eg}
Let $C$ be a strictly semistable curve of genus $g$ over an algebraically closed field $k$ 
and $C'$ its stabilization.  We attach the following discrete data 
to $C$.
We let $m$ be the number of positive dimensional fibers of the stabilization $C' \to C$,
$n$ be the maximum
number of irreducible components of such a fiber, and  $\br=(r_1,\ldots,r_n)$ be the sequence
of nonnegative integers where $r_i$ is the number of such fibers having precisely $i$ irreducible components.
Note that 
$r_n > 0$ and $\sum_i r_i = m$; moreover, the number of exceptional components
is $\sum_i i r_i$.

We define 
\[
E_g^{\times m} = \underbrace{E_g\times_{\overline{\mathcal{M}}_g}
E_g\times_{\overline{\mathcal{M}}_g}
\cdots
\times_{\overline{\mathcal{M}}_g}E_g}_{m \text{ times}} \\
\]
Similarly, we denote by $X_{g, \triangle}^{\times m}$ 
the $m$-fold fiber product of $X_{g, \triangle}$
over $\bar{\mathcal{M}}_g$.
Define $(E_g^{\times m})^{\circ} \subset E_g^{\times m}$ as the open substack 
parametrizing curves where no pair of selected nodes is the same, and let 
  $(X_{g,\triangle}^{\times m})^{\circ}$ 
be the pre-image of $(E_g^{\times m})^{\circ}$.

The product $ S_{\br}=S_{r_1}\times\cdots\times S_{r_n}$ of symmetric groups
acts naturally on the $m$-fold fiber products.  
Evidently, $(E_g^{\times m})^{\circ}$ and $(X_{g,\triangle}^{\times m})^{\circ}$ 
are $S_{\br}$-invariant.  There is a point 
$x' \in (X_{g,\triangle}^{\times m})^{\circ}$ corresponding
to the $m$-tuple of semistable curves $(C_1, \ldots, C_m)$, 
where $C_i$ is obtained by 
contracting all exceptional components which do not lie  
over the $i$th marked point of the stabilization.

\begin{prop} \label{prop-map-to-mgss}
The morphism $(X_{g,\triangle}^{\times m})^{\circ} \to \mathfrak{M}^{\ss}_g$, 
defined by mapping
an $m$-tuple $(B_1, \ldots, B_m)$ of semistable curves with 
stabilization $B'$ to the fiber product
\[ B_1 \times_{B'} \cdots \times_{B'} B_m, \]
is \'etale, representable and $S_{\br}$-equivariant.  
Moreover, the induced morphism 
$$[(X_{g, \triangle}^{\times m})^{\circ} / S_{\br} ] \to \mathfrak{M}^{\ss}_g$$
is stabilizer preserving at $x'$ and the diagram
$$\xymatrix{
[(X_{g, \triangle}^{\times m})^{\circ} / S_{\br} ] \ar[r] \ar[d]	
	& [(E_g^{\times m})^{\circ} / S_{\br} ] \ar[d] \\
\mathfrak{M}^{\ss}_g \ar[r]						
	& \bar{\cM}_g
}$$
is 2-commutative with \'etale vertical morphisms and representable induced morphism
$[(X_{g, \triangle}^{\times m})^{\circ} / S_{\br} ]\to
\mathfrak{M}^{\ss}_g\times_{\bar{\cM}_g}[(E_g^{\times m})^{\circ} / S_{\br} ]$.
\end{prop}

\begin{proof}
 Since $(E_g^{\times m})^{\circ} \subset  (X_{g, \triangle}^{\times m})^{\circ}$ has 
 complement of codimension 2 and $X_{g, \triangle} \to \bar{\mathcal{M}}_g$ is
 Cohen-Macaulay,  $(X_{g, \triangle}^{\times m})^{\circ}$ is normal.
 Since the morphism $(X_{g, \triangle}^{\times m})^{\circ} \to \mathfrak{M}^{\ss}_g$
 is \'etale in codimension 1, we may conclude from Zariski-Nagata purity \cite[Thm.\ X.3.1]{sga1} 
that it
is \'etale.  The remaining statements are clear.
\end{proof}

Since the morphism $[(X_{g, \triangle}^{\times m})^{\circ} / S_{\br} ] \to \mathfrak{M}^{\ss}_g$
is not in general representable, some refinement is necessary in order to
reduce the verification of Conjecture \ref{conj1} to a stack related to
$[(X_{g, \triangle}^{\times m})^{\circ} / S_{\br} ]$.
We now provide such a refinement.

We recall the degree $2$ cover $\widehat{D}_g\to D_g$ from \S \ref{subsec:boundary-Mg}.
By \eqref{eqn:A1ringF}, there is a fiber diagram
\begin{equation}
\begin{split}
\label{eqn:FEA1}
\xymatrix{
F_{\widehat{D}_g/E_g} \ar[r] \ar[d] & \mathcal{A}^1 \ar[d] \\
E_g \ar[r] & \mathring{\mathcal{A}}^1
}
\end{split}
\end{equation}
where the bottom arrow is given by $\widehat{D}_g\to D_g$.
We let $F_{\widehat{D}_g/E_g}^{\times m}$ be the $m$-fold fiber product
over $\bar{\mathcal{M}}_g$ and
$(F_{\widehat{D}_g/E_g}^{\times m})^{\circ}$
the pre-image of $(E_g^{\times m})^{\circ}$.

\begin{prop} \label{prop:second}
There is an \'etale morphism $[\Spec(A)/H] \to [(E_g^{\times m})^{\circ} / S_{\br} ]$,
stabilizer preserving at a point $y''$ over the image in $[(E_g^{\times m})^{\circ} / S_{\br} ]$ of $x' \in (X_{g, \triangle}^{\times m})^{\circ}$ such that the composition 
$$[\Spec(A)/H] \to [(E_g^{\times m})^{\circ} / S_{\br} ] \to \bar{\cM}_g$$
is representable.
Moreover, there is a finite \'etale cover $\Spec(B) \to [\Spec(A) / H]$ such that
the composition $\Spec(B) \to [\Spec(A) / H] \to [(E_g^{\times m})^{\circ} / S_{\br}]$
factors through $(F_{\widehat{D}_g/E_g}^{\times m})^{\circ}$.
\end{prop}

\begin{proof}
Let $W_{\br}=W_{r_1}\times\cdots\times W_{r_n}$ be the
product of the hyperoctohedral groups of signed permutations.  There is an action
of $W_{\br}$ on $(F_{\widehat{D}/E_g}^{\times m})^{\circ}$  compatible with the $S_{\br}$-action
on $(E_g^{\times m})^{\circ}$; this gives an \'etale morphism 
$[(F_{\widehat{D}/E_g}^{\times m})^{\circ} / W_{\br}] \to [(E_g^{\times m})^{\circ} / S_{\br}]$
which is stabilizer preserving at the unique pre-image $\tilde{y}$ in 
$[(F_{\widehat{D}/E_g}^{\times m})^{\circ} / W_{\br}]$
of the image of $x' \in (X_{g, \triangle}^{\times m})^{\circ}$ 
in $[(E_g^{\times m})^{\circ} / S_{\br}]$.
Let $H$ denote the stabilizer group at $\tilde{y}$.  By Lemma \ref{lem:representablelocus} and 
\S \ref{subsec:localDM}, there exists
an affine scheme $\Spec(A)$ with an action of $H$ and an \'etale morphism
$[\Spec(A)/H] \to  [(F_{\widehat{D}/E_g}^{\times m})^{\circ} / W_{\br}]$ stabilizer preserving at a point $y'' \in [\Spec(A) / H]$ above $\tilde{y}$
such that the composite morphism
$[\Spec(A) / H] \to \bar{\cM}_g$ 
is \'etale and representable.  

The final statements are obtained by taking $\Spec(B)$ to be the affine scheme representing the base change 
$\Spec(A) 
	\times_{[(F_{\widehat{D}/E_g}^{\times m})^{\circ} / W_{\br}]} 
	(F_{\widehat{D}/E_g}^{\times m})^{\circ}$.  
\end{proof}

\begin{cor}
\label{cor:second}
With the notation of Proposition \ref{prop:second},
if we define $U$ as the fiber product
\[
\xymatrix{
U \ar[r] \ar[d]		&  [\Spec(A)/H] \ar[d] \\
[(X_{g, \triangle}^{\times m})^{\circ} / S_{\br} ] 	\ar[r]		&
	[(E_g^{\times m})^{\circ} / S_{\br} ]
}
\]
then the composition $U \to [(X_{g, \triangle}^{\times m})^{\circ} / S_{\br} ]  \to \fM^{\ss}_g$ is \'etale and representable, and stabilizer preserving at a pre-image $x''$ of $C$.
\end{cor}

\section{Local structure of $X_g$} \label{sec:locXg}

In the section, we exhibit an augmented semi-aligned log structure on $X_g$.
This semi-aligned log structure allows us to establish Proposition \ref{prop:fiberdiagramxtriangle} which provides an explicit description of the fiber of the stabilization morphism over a stable curve with precisely one node.
This proposition will in turn be applied in \S \ref{sec:proof} to complete the proof of Theorem \ref{main-theorem}.

\subsection{Augmented semi-aligned log structure on $X_g$}
\label{subsec:semialignedXg}

Recall that $X_g$ was defined in \S \ref{subsec:stabilization}.
We denote by $\widehat{X}_g$ the fiber product
\[
\xymatrix{
\widehat{X}_g \ar[r] \ar[d]		&  \widehat{D}_g \ar[d] \\
X_g \ar[r]^{\st}				& E_g
}
\]
where the degree 2 \'etale cover 
$\widehat{D}_g \to D_g \subset E_g$ 
is as in \S \ref{subsec:boundary-Mg}.
By applying the construction of \S \ref{subsec:etalecoverings},
diagram \eqref{eqn:FEA1} may be extended to the left with a fiber square
\begin{equation}
\begin{split}
\label{eqn:extend}
\xymatrix{
F_{\widehat{X}_g/X_g} \ar[r] \ar[d] &
F_{\widehat{D}_g/E_g} \ar[d] \\
X_g \ar[r] & E_g
}
\end{split}
\end{equation}
We observe,
$F_{\widehat{D}_g/E_g}\to \overline{\mathcal{M}}_g$ and $\widehat{D}$
satisfy the 
hypotheses of Proposition \ref{prop:alignedlog} and thus determine
an aligned log structure on
$F_{\widehat{X}_g/X_g}$.
Now Proposition \ref{prop:augmentedsemialigned} may be applied to yield
an augmented semi-aligned log structure
\begin{equation}
\label{eqn:semialignedonX}
X_g\to \mathcal{L}og^{(\frac{1}{2}+\varepsilon)\mathrm{al}}
\end{equation}
fitting into a fiber diagram
\begin{equation}
\begin{split}
\label{eqn:FXhataligned}
\xymatrix{
F_{\widehat{X}_g/X_g}\ar[r] \ar[d] & \mathcal{L}og^{\mathrm{al}} \ar[d] \\
X_g \ar[r] & \mathcal{L}og^{(\frac{1}{2}+\varepsilon)\mathrm{al}}
}
\end{split}
\end{equation}
with the morphism corresponding to the aligned log structure on
$F_{\widehat{X}_g/X_g}$ and vertical \'etale morphisms.

\subsection{Structure of $X_{g, \triangle}$}
\label{subsec:strXgtriangle}
The next result is an analogue of
Lemma \ref{lem:localstr}
for $X_{g, \triangle}$.

\begin{prop} \label{prop:fiberdiagramxtriangle}
There is a fiber diagram
\begin{equation}
\begin{split}
 \label{diagram-xtriangle}
\xymatrix{
X_{g, \triangle} \ar[r]^{\st} \ar[d]			& E_g \ar[d] \\
\mathcal{L}og^{(\frac{1}{2}+\varepsilon)\mathrm{al}}
\ar[r] & \mathring{\cA}^1
}
\end{split}
\end{equation}
where the left vertical arrow is the restriction to $X_{g, \triangle}$
of the morphism \eqref{eqn:semialignedonX}
and the right vertical arrow is given by the
finite \'etale cover $\widehat{D}_g \to D_g$.
\end{prop}

\begin{proof}
Since the morphisms $X_{g, \triangle} \to E_g$ and 
$\mathcal{L}og^{(\frac{1}{2}+\varepsilon)\mathrm{al}} \to \mathring{\cA}^1$
are isomorphisms in 
codimension 1, so is the morphism
\begin{equation} \label{eqn-psi}
\Psi \colon X_{g, \triangle} \to \mathcal{L}og^{(\frac{1}{2}+\varepsilon)\mathrm{al}}
\times_{ \mathring{\cA}^1} E_g.
\end{equation}
The fiber product in \eqref{eqn-psi} is smooth
(since the right-hand morphism in \eqref{diagram-xtriangle} is smooth),
so it follows from Zariski-Nagata purity \cite[Thm.\ X.3.1]{sga1}
that $\Psi$ is \'etale.  
Since $\Psi$ is also bijective on geometric points and is 
pointwise stabilizer preserving, we conclude from 
Lemma \ref{lem:etale} that $\Psi$ is an isomorphism.
\end{proof}

\section{Proof of Theorem \ref{main-theorem}} \label{sec:proof}

\begin{thm} Conjecture \ref{conj1} holds for $\mathfrak{M}_g^{\ss}$.
\end{thm}

\begin{proof}  Let $C$ be a strictly semistable curve with stabilization $C'$.  Let $m$ and $\br=(r_1,\ldots,r_n)$ be the combinatorial data assigned to $C$ as introduced in \S \ref{subsec:eg}.

We apply Proposition \ref{prop:second} and, with the notation from there, define
\begin{gather*}
X':=[\mathring{\cA}^{\times \br} / S_{\br}]
\times_{[(\mathcal{L}og^{(\frac{1}{2}+\varepsilon)\mathrm{al}})^{\times m}/S_{\br}]}
[(X_{g, \triangle}^{\times m})^{\circ} / S_{\br}]
\times_{[(E_g^{\times m})^{\circ}  / S_{\br}]} [\Spec(A)/H],\\
X'':=[\mathring{\cA}^{\times \br} / S_{\br}]
\times_{[(\mathcal{L}og^{(\frac{1}{2}+\varepsilon)\mathrm{al}})^{\times m}/S_{\br}]}
[(X_{g, \triangle}^{\times m})^{\circ} / S_{\br}]
\times_{[(E_g^{\times m})^{\circ}  / S_{\br}]} \Spec(B).
\end{gather*}
Notice, with the notation of Corollary \ref{cor:second},
$X'$ may be identified with the fiber product
$[\mathring{\cA}^{\times \br} / S_{\br}]
\times_{[(\mathcal{L}og^{(\frac{1}{2}+\varepsilon)\mathrm{al}})^{\times m}/S_{\br}]}U$
and hence (cf.\ Proposition \ref{prop:Araugmentedsemialigned}) admits a
morphism to $\fM_g^{\ss}$ satisfying conditions (iii) and (iv) of
Proposition \ref{prop:finiteetale}.
The cover $X''\to X'$ satisfies condition (ii).
So it remains to verify condition (i).

By Proposition \ref{prop:second}, $X''$ is isomorphic to
\[
\mathring{\cA}^{\times \br}
\times_{(\mathcal{L}og^{(\frac{1}{2}+\varepsilon)\mathrm{al}})^{\times m}}
(F_{\widehat{X}_{g, \triangle}/X_{g, \triangle}}^{\times m})^{\circ}
\times_{(F_{\widehat{D}_g/E_g}^{\times m})^{\circ}}\Spec(B), \]
where
$\widehat{X}_{g, \triangle}$ denotes the pre-image of $X_{g, \triangle}$ under
$\widehat{X}_g\to X_g$ and where
$(F_{\widehat{X}_{g, \triangle}/X_{g, \triangle}}^{\times m})^{\circ}$
denotes the pre-image of $(E_g^{\times m})^{\circ}$ in the
$m$-fold fiber product
$F_{\widehat{X}_{g, \triangle}/X_{g, \triangle}}^{\times m}$
over $\overline{\mathcal{M}}_g$.
Thanks to \eqref{eqn:FXhataligned} and
Proposition \ref{prop:Araugmentedsemialigned}, $X''$ may be identified with
\[ \cA^{\times \br}
\times_{(\mathcal{L}og^{\mathrm{al}})^{\times m}}
(F_{\widehat{X}_{g, \triangle}/X_{g, \triangle}}^{\times m})^{\circ}
\times_{(F_{\widehat{D}_g/E_g}^{\times m})^{\circ}}\Spec(B). \]

The next step is to obtain from diagrams
\eqref{eqn:FEA1}, \eqref{eqn:extend}, and
\eqref{eqn:FXhataligned}
and Propositions \ref{prop:Araugmentedsemialigned} and
\ref{prop:fiberdiagramxtriangle} the fiber diagram
\[
\xymatrix{
F_{\widehat{X}_{g, \triangle}/X_{g, \triangle}} \ar[r] \ar[d] &
F_{\widehat{D}_g/E_g} \ar[d] \\
\mathcal{L}og^{\mathrm{al}} \ar[r] & \mathcal{A}^1
}
\]
This yields further fiber diagrams
upon passing to $m$-fold fiber products (in the top row,
over $\overline{\mathcal{M}}_g$), and upon restricting to
pre-images of $(E_g^{\times m})^{\circ}$.
Now we have
\[ X''\cong \cA^{\times \br}\times_{(\mathcal{A}^1)^{\times m}}\Spec(B), \]
and we conclude, as in the proof of Theorem \ref{main-theorem-special-case},
by observing that the morphism from $\Spec(B)$
to $(\mathcal{A}^1)^{\times m}$ is
affine.
\end{proof}

\bibliography{refs}{}
\bibliographystyle{amsplain}

\end{document}